\documentclass{amsart}
\newtheorem{theorem}{Theorem}[section]
\newtheorem{lemma}[theorem]{Lemma}
\theoremstyle{definition}
\newtheorem{definition}[theorem]{Definition}
\newtheorem{example}[theorem]{Example}

\newtheorem{corollary}[theorem]{Corollary}
\newtheorem{proposition}[theorem]{Proposition}
\theoremstyle{remark}
\newtheorem{remark}[theorem]{Remark}

\numberwithin{equation}{section}

\begin{document}
\def\C{\mathbb C}
\def\D{\mathbb D}
\def\R{\mathbb R}
\def\X{\mathbb X}
\def\cA{\mathcal A}
\def\cT{\mathcal T}
\def\Z{\mathbb Z}
\def\Y{\mathbb Y}
\def\Z{\mathbb Z}
\def\N{\mathbb N_0}
\def\cal{\mathcal}
\def\F{\mathcal F}

\title[Asymptotic Behavior of Volterra Equations]{On the Asymptotic Behavior of Volterra Difference Equations}
\author{Nguyen Van Minh}
\address{Department of Mathematics, University of West Georgia,
Carrollton, GA 30118} \email{vnguyen@westga.edu}

\thanks{The author would like to thank the referees for carefully reading the manuscript, and for remarks and suggestions to improve the previous version of the paper.}

%    General info

%%%%%%%%%%%%%% ABSTRACT %%%%%%%%%%%%%%%%%%%%%%%%%%%%%%%
\begin{abstract}
We consider the asymptotic behavior of solutions of the
difference equations of the form $x(n+1)=Ax(n) + \sum_{k=0}^n B(n-k)x(k) + y(n)$ in a Banach
space $\X$, where $n=0,1,2,...$; $A,B(n)$ are linear bounded operator
in $\X$. Our method of study is based on the concept of spectrum of a unilateral sequence. The obtained results on asymptotic stability and almost periodicity are stated in terms of spectral properties of the equation and its solutions. To this end, a relation between the Z-transform and spectrum of a unilateral sequence is established. The main results extend previous ones.
\end{abstract}
\date{\today}

\maketitle

%%%%%%%%%%%%%% INTRODUCTION %%%%%%%%%%%%%%%%%%%%%
\section{Introduction, Notations and Preliminaries} \label{section 1}
\subsection{Introduction}
The purpose of this note is to study the asymptotic behavior of solutions to
difference equations of the forms
\begin{eqnarray}
x(n+1)&=&Ax(n) + \sum_{k=0}^n B(n-k)x(k),  \quad x(n)\in \X, n\in \N ,\label{eq}\\
x(n+1)&=&Ax(n) + \sum_{k=0}^n B(n-k)x(k) + y(n),  \quad x(n)\in \X, n\in \N ,\label{eq1}
\end{eqnarray}
where $x(n)\in\X$, $A,B(n)$ are linear continuous operators acting in a Banach space $\X$, $y(n)\in  \X$ is a bounded sequence, under the condition
\begin{equation}\label{bounded eq}
\sum_{k=0}^\infty \| B(k)\| < \infty .
\end{equation}
The asymptotic behavior of solutions of (\ref{eq}) and (\ref{eq1}) is a classical subject of Dynamical Systems and Operator Theory. We refer the reader to the books \cite{arebathieneu,ela}, and papers \cite{arebat3,bat,chitom,elamur,eststrzou2,kalmonoletom,kattza,lauvu,min,naiminmiyham,vu}
and their references for more information in this direction. Related results for differential equations can be found in \cite{arebat,arepru,baspry,min2,minngusie}.

\medskip
In this note we will study Volterra equations of convolution type as a continuation of the research begun in \cite{min}, in which we introduced the concept of spectrum of a "unilateral" sequence, and proposed a new treatment and generalization of some classical results such as Katznelson-Tzafriri Theorem and Arendt-Batty-Ljubich-Vu Theorem for discrete systems. Volterra equations of convolution type
serve as a model for evolutionary processes that take into account of processes' history. This is why Volterra equations are often met in applications. We will combine the spectrum method with the  method of Z-transform that is widely popular in the study of discrete Volterra equations. We will prove various versions of the above mentioned theorems for Volterra equations of the form (\ref{eq}) and (\ref{eq1}) that are stated in terms of spectral properties of equations and solutions. To this end, after recalling some concepts and notations in Section 2 we will discuss the relation between the spectrum of a sequence and its Z-transform (Section 3). In Section 4 we will apply the results of Section 3 to Eq. (\ref{eq}) and (\ref{eq1}) after estimating the spectrum of a bounded solution
using Z-transform and induced operators in the quotient space $\Y := l^\infty (\X) /c_0(\X)$.

\subsection{Notations}
 In this note we will use the following notations: $\N =\{ 0,1,2,\cdots \}$, $\Z$ -  the set of all integers, $\R$ - the set of reals, $\C$ - the complex plane with $\Re z$ denoting the real part of $z\in\C$, $\X$ - a given  complex Banach space. $\Gamma$ and  $ \mathbb D$
denote the unit circle and the open unit disk in the complex plane.
  A sequence in $\X$ will be denoted by $\{x(n)\}_{n=0}^\infty$, or, simply by $(x(n))$, and the spaces of sequences
\begin{eqnarray*}
l^\infty (\X)  :=  \{ (x(n))  \subset \X | \
\sup_{n\in\N} \| x (n)\| <\infty \} , \
c_0 (\X) :=  \{ (x(n))\subset \X | \ \lim_{n\to\infty}x(n) =0\}
\end{eqnarray*}
are equipped with sup-norm. If $A:=(A(n))$ and $y:= (y(n))$ are two sequences in $l^\infty (L(\X))$ and $l^\infty (\X)$, respectively, then, the sequence
$
(A*y)(n):= \sum_{k=0}^n A(k)y(n-k), n\in \N,
$
is called the convolution of $A$ and $y$. If $B\in L(\X)$, then ${\cal B}$ is defined to be the operator in $l^\infty (\X)$ that maps $(x(n))$ to $(Bx(n))$.
The translation operator $S(k)$ acts in $l^\infty (\X)$ as follows:
$
[S(k)x](n) = x(n+k),\quad n\in\N , x\in l^\infty (\X) .
$
For simplicity $S(1)$ will be denoted by $S$. In this paper the space of all
bounded linear operators acting in $\X$ is denoted by $L(\X)$; $\rho
(B), \sigma (B), R\sigma (B), Ran (B)$ denote  the resolvent set,
spectrum, residual spectrum, range of $B\in L(\X)$, respectively. Here we use the definition of the {\it residual spectrum} $R\sigma (A)$ of an operator $A$ in a Banach space $\X$ as the  set of  all $\lambda \in \C$  such  that  range $Ran (\lambda -A)$  is not  dense  in  $\X$.
The notation $B(x_0,r)$ will stand for $\{ x\in \X: \| x-x_0\| <r\}$, and its closure is denoted by $\bar B(x_0,r)$.
\subsection{Asymptotically almost periodic sequences}
Recall that an almost periodic sequence on $\Z$ is
defined to be an element of the subspace
$$AP(\Z,\X):= \overline{span\{
(\lambda ^{n} y_0)_{n\in \mathbb Z},\lambda \in \Gamma ,y_0 \in
\mathbb X\}}$$
of the space of all bounded bilateral sequences $l^\infty (\Z,\X)$ equipped with sup norm.

Let us denote by $AP(\N,\X)$ the space of all restrictions of almost periodic bilateral sequences to $\N$. It can be proved that each element of $AP(\N,\X)$ is the restriction of exactly one element of $AP(\Z,\X)$.
$AAP(\N, \X)$ denotes the space of all asymptotically almost periodic sequences $( x(n))^\infty _{n=0} \subset \X$.
Recall that a sequence $(x(n))\in l^\infty (\X)$ is said to be asymptotically almost
periodic if $x(n)=y(n)+z(n)$ for all $n\in \N$, where $(y(n))\in c_0(\X)$
and $(z(n))$ is an almost periodic sequence.

\subsection{Vector-valued holomorphic functions} The reader is referred to a brief introduction to vector-valued holomorphic functions in \cite[Appendix A, pp. 455-459]{arebathieneu} for a definition and basic properties of holomorphic functions.
A family of continuous functionals $W \subset \X^*$ is said to be {\it separating} if $x\in \X$ and $<x,\phi >=0$ for all $\phi \in W$, then $x=0$. We will need the following whose proof can be found in \cite[Theorem A.7]{arebathieneu}:
\begin{theorem}
Let $\Omega \subset \C$ be an open and connected, and let $f:\Omega \to \X$ be bounded on every compact subset of $\Omega$. Assume further that $W\subset \X^*$ is separating subspace such that $x^* \circ f$ is holomorphic for all $x^*\in W$. Then, $f$ is holomorphic.
\end{theorem}
We will need an auxiliary result that is a special kind of maximum principle for holomorphic functions (for the proof see e.g. \cite[Lemma 4.6.6]{arebathieneu}):
\begin{lemma}\label{lem 4.6.6}
Let $U$ be an open neighborhood of $i\eta$ such that $U$ contains the closed disk $\bar{B}(i\eta ,2r)=\{ z\in\C :|z-i\eta | \le 2r\}$. Let $h:U\to \X$ be holomorphic and $c\ge 0$, $k\in \N$ such that
\begin{equation}
\| h(z)\| \le \frac{c}{|\Re z|^k}, \quad \mbox{if} \ |z-i\eta |=2r , \Re z \not = 0.
\end{equation}
Then
\begin{equation}
\| h(z)\| \le \left( \frac{4}{3}\right)^k  \frac{c}{r^k}, \quad \mbox{for all}\ \ z\in \bar B(i\eta ,r).
\end{equation}
\end{lemma}

\section{Spectral Theory of Unilateral Sequences}
Let $\cal F$ be a non-trivial closed subspace of $l^\infty (\X)$ that is translation-bi-invariant, that is, $S^{-1}({\cal F})={\cal F}$. A model for such spaces ${\cal F}$ is $c_0(\X)$ or $AAP(\N,\X)$. It is easy to see that this definition is equivalent to $S^{-n}({\cal F})={\cal F}$ for any positive integer $n$, so any sequence $\{ x(n)\}_{n=0}^\infty$ such that $x(n)=0$ for all large $n$ belongs to ${\cal F}$. Therefore,
every translation-bi-invariant subspace of $l^\infty (\X)$ contains $c_0(\X)$. For related studies on this kind of spaces the reader is referred to \cite{arebat,arebathieneu}. Consider the quotient space $\Y := l^\infty (\X) /\F $
with the induced norm. The equivalent class containing a sequence
$x:=(x(n))\in l^\infty (\X)$ will be denoted by $\bar x$. Since $S$ leaves
$\F$ invariant it induces a bounded linear operator $\bar S$ acting
in $\Y$. Moreover, one notes that $\bar S$ is a surjective isometry.
As a consequence, $\sigma (\bar S)\subset \Gamma$. We will use the
following estimate for the resolvent of the isometry $\bar S$ whose
proof can be easily obtained:
\begin{equation}\label{112}
\| R(\lambda ,\bar S)\| \le \frac{1}{||\lambda |-1|}, \quad \mbox{for all} \ |\lambda | \not= 1.
\end{equation}

\begin{definition}
Let $(x(n))$ be a bounded sequence in $\X$.
The notation $\sigma (x)$ stands
for the set of all non-removable singular points of the complex
function $g(\lambda ):= R(\lambda ,\bar S)\bar x$. This set will be denoted by $\sigma_\F (x)$, and
referred to as {\it the reduced spectrum of $x$} with respect to $\F$, or simply $\F$-spectrum of $x$. If $\F=c_0(\X)$ we will denote the $\F$-spectrum of $x:=(x(n))\in l^\infty (\X)$ by $\sigma (x)$, and call it the spectrum of $x$.
\end{definition}

In the following we summarize some important properties of the spectrum of a function.
\begin{proposition}\label{pro 3.3}
Let $\{ g_n\}_{n=1}^\infty \subset l^\infty (\X)$ such that
$g_n\to g\in l^\infty (\X)$, and let $\Lambda$ be a closed subset
of the unit circle. Then the following assertions hold:
\begin{enumerate}
\item \ $\sigma (g)$ is closed. \item \ If $\sigma (g_{n}) \subset
\Lambda$ for all $n \in {\N}$, then $\sigma (g)\subset \Lambda $.
\item \ $\sigma ({\cal A}g)\subset \sigma (g)$ for all $A\in
L(\X)$. \item  $\sigma (g) =\emptyset$, if and only if $\bar g=\bar 0$.
\end{enumerate}
\end{proposition}
\begin{proof}
The properties (i) and (ii) are obvious from the definition. The proof of (ii) can be found in \cite[Proposition 3.3]{minngusie}, and (iv) is proved
in \cite[Lemma 2.7]{min}.
\end{proof}
\begin{corollary}
Let $\Lambda$ be a closed subset of the unit circle $\Gamma$. Then, the set $\Lambda (\X)$ of all sequences $x:=(x(n))\in l^\infty (\X)$ whose spectrum $\sigma (x) \subset \Lambda$ is a closed linear subspace of $l^\infty (\X)$.
\end{corollary}
\begin{proof}
The corollary is an immediate consequence of part (ii) of proposition \ref{pro 3.3}.
\end{proof}

 We define ${\cal M}_{\bar x}$ as the smallest closed subspace
of $\Y := l^\infty (\X)/c_0$ spanned by $\{ \bar S^n\bar x, n\in
\Z\}$. Consider the restriction $\bar S |_{{\cal M}_{\bar x}}$ that
is also a surjective isometry. The following lemma was actually proven in
\cite[Lemma 2.7]{min}, so its proof is omitted:
\begin{lemma}\label{lem 4} Let $x:=(x(n))\in l^\infty (\X )$. Then, the following assertions hold:
\begin{enumerate}
\item $\sigma (x)=\emptyset$ if and only if $x\in c_0$;
\item If $\sigma ( x) \not= \emptyset$, then
  $\sigma (x)=\sigma ( \bar S |_{{\cal M}_{\bar x}})$.
\end{enumerate}
\end{lemma}

Before proceeding we introduce a new notation:  let $0\not= z\in \C$ such that
$z= re^{i\varphi}$ with reals $r=|z|,\varphi$, and let $F(z)$ be any complex
function. Then, (with $s$ larger than $r$) we define
\begin{equation}
\lim_{\lambda \downarrow z} F(\lambda ):= \lim_{s\downarrow r}F(s e^{i\varphi}).
\end{equation}
The following theorem will be the key tool for our results.
\begin{theorem}\label{the tec}
Let $(x(n))$ be a bounded sequence such that $\sigma ( x)$
 is countable, and let the following condition hold for
each $\xi_0\in \sigma (x)$
\begin{equation}
\lim_{\lambda \downarrow \xi_0} (\lambda -\xi_0 )R(\lambda ,\bar
S)\bar x =0 .
\end{equation}
Then, \begin{equation}\label{55} \lim_{n\to\infty}x(n)=0.
\end{equation}
\end{theorem}
\begin{proof}
For the proof see \cite[Theorem 2.8]{min}.
\end{proof}

\section{The Z-transform and spectrum of a sequence}

Recall that the $Z$-transform of a sequence  $x:=\{ x(n)\}_{n=0} ^\infty$ is defined as
\begin{equation}
\tilde{x}(z):=\sum_{j=0}^\infty x(j)z^{-j}.
\end{equation}
\begin{proposition}
Let $(A(n))\in l^\infty (L(\X))$ and $(x(n))$ be in $l^\infty (\X)$. Then,
\begin{enumerate}
\item $\tilde{x}(z)$ is a complex function in $z$ that is defined and holomorphic for $|z|>1$;
\item $\widetilde{Sx}(z)= z\tilde{x}(z)-zx(0)$;
\item $\widetilde{A*x}(z)=\tilde{A}(z)\cdot \tilde{x}(z)$ if in addition $\sum _{n=0}^\infty \| A(n)\| <\infty $;
\item For each $x:=(x(n))\in l^\infty (\X)$
\begin{equation}\label{initial value of z-trans}
\lim_{|z|\to\infty}\tilde{x}(z)=x(0).
\end{equation}
\end{enumerate}
\end{proposition}
\begin{proof}
For the proof see e.g. \cite[Chap. 5]{ela}.
\end{proof}

\begin{lemma}\label{lem 3.1}
For a given $x:=(x(n))\in l^\infty (\X)$ let $z_0\in \Gamma$ such that $R(z,\bar S)\bar x$, as a complex function of $z$ in the domain $D_{|z|>1}:=\{ z\in\C : |z|>1\}$, has an holomorphic extension $h(z)$ to a neighborhood of $z_0\in \Gamma$. Then, $z_0$ is not in $\sigma (x)$.
\end{lemma}
\begin{proof}
Consider the function 
$$f(z):= (z-\bar S)h(z)$$
in a connected neighborhood $U(z_0)$ of $z_0$. This is a holomorphic function in $U(z_0)$. Moreover, for $z\in U(z_0) \cap D_{|z|>1}$
$$f(z)= (z-\bar S)h(z)=(z-\bar S) R(z,\bar S)\bar x =\bar x.$$
Therefore, $f(z)$ must be a constant function on $U(z_0)$. And thus, for all $z\in U(z_0)$ such that $|z|<1$ we have
$(z-\bar S)h(z)=\bar x$, this is equivalent to say that for all $z\in U(z_0)$ such that $|z|<1$
$$
h(z)=R(z,\bar S)\bar x  .
$$
This yields that $z_0$ is a regular point of $R(z,\bar S)\bar x$, so not in $\sigma ( x)$. 
\end{proof}

\begin{lemma}\label{lem z-transform}
Let $x:=(x(n))\in l^\infty (\X)$. If the Z-transform $\tilde{x}(z)$ of the sequence $x$ has an holomorphic extension to a neighborhood of $z_0\in\Gamma$, then $z_0\not\in \sigma (x).$
\end{lemma}
\begin{proof} Assume that $\tilde{x}(z)$ (with $|z|> 1$) can be extended to a holomorphic function $g_0(z)$ in $B(z_0,\delta )$ with a sufficiently small positive $\delta$. We will show that $R(z,S)x$ (with $|z|>1$) has an holomorphic extension in a neighborhood of $z_0$.
For $|z|>1$ we have
\begin{eqnarray}
R(z,S)x &:=& (z-S)^{-1}x
  =  z^{-1}(I-z^{-1}S)^{-1}x \nonumber  \\
& =& z^{-1}\sum_{n=0}^{\infty}z^{-n}S(n)x
  = \sum_{n=0}^{\infty}z^{-n-1}S(n)x \nonumber
\end{eqnarray}
For each $k\in \N$ and $z\in B(z_0,\delta ) $, $|z|>1$, by the definition of the Z-transform
\begin{eqnarray}
[R(z,S)x] (k) & =& \sum_{n=0}^{\infty}z^{-n-1}x(n+k)
 =  z^{-1} \widetilde{S(k)x}(z) \label{3.100}\\
&=&  z^{-1}  \left( z^k\tilde {x}(z)-\sum_{j=0}^{k-1} z^{k-j}x(j)    \right)  \nonumber\\
&=& z^{k-1}g_0(z) -\sum_{j=0}^{k-1} z^{k-1-j}x(j).\nonumber
 \end{eqnarray}
 Below we will denote
 \begin{equation}
 g_k(z):=z^{k-1}g_0(z) -\sum_{j=0}^{k-1} z^{k-1-j}x(j), \quad z\in B(z_0,\delta ).
 \end{equation}
 For each $k\in\N$, the complex function $[R(z,S)x] (k)$ (with $|z|>1$) can be extended to a complex function $g_k(z)$ in $B(z_0,\delta )$. To show that the function $R(z,S)x$ (with $|z|>1$) is extendable to a holomorphic function in  a neighborhood of $z_0$ we first prove that
 \begin{equation}\label{boundedness of g(z)}
  \sup_{z\in B(z_0,r)} \| g(z)\| <\infty ,
 \end{equation}
 where $r$ is a certain small positive number, and the function $g(z)$ is defined as
 \begin{equation}
 B(z_0,\delta )\ni z \mapsto g(z):= \{ g_k (z)\} _{k=0}^\infty \in l^\infty (\X).
 \end{equation}
Note that for all $k\in\N$, $g_k$ is holomorphic in $B(z_0,\delta )$.
If $z\in B(z_0,\delta)$ and $|z|>1$, then
\begin{eqnarray}
\|g_k(z)\|&=& \| [R(z,S)x] (k)\|
  =  \| z^{-1} \widetilde{S(k)x}(z)\| \nonumber\\
&\le&  | z^{-1} | \cdot \sum_{j=0}^\infty |z|^{-j}\|x(j+k)\| \nonumber\\
&\le&   | z^{-1} | \cdot \sum_{j=0}^\infty |z|^{-j} | \cdot \| x\|
 =  \frac{|z|^{-1}}{1-|z|^{-1}}\| x\| \nonumber\\
&=& \frac{1}{|z|-1 }\| x\| \label{3.14}
\end{eqnarray}
If if $z\in B(z_0,\delta )$ and $|z|<1$, then, for each $k\in\N$,
\begin{eqnarray}
\|g_k(z)\| &=&\|
 z^{k-1}g_0(z) -\sum_{j=0}^{k-1} x(j)z^{k-1-j}\| \nonumber\\
 &\le& \sup_{\lambda \in B(z_0,\delta  )} \| g_0(\lambda)\| \cdot |z|^{k-1}+\sum_{j=0}^{k-1}|z|^{k-1-j}\|x\| \nonumber \\
 &\le& \left(\sup_{\lambda \in B(z_0,\delta_  )} \| g_0(\lambda)\| +\| x\| \right) \sum_{j=0}^\infty |z|^j\nonumber\\
 &=&\frac{1}{1-|z|}\left(\sup_{\lambda \in B(z_0,\delta  )} \| g_0(\lambda)\| +\| x\| \right)\label{3.15}
\end{eqnarray}
Now (\ref{3.14}) combined with (\ref{3.15}) shows that  there is a positive constant $C$ independent of $n$ and $z$ such that if $z\in B(z_0,\delta)\backslash \Gamma$, then
\begin{equation}\label{3.10}
\| g_n (z)\| \le \frac{C}{|1-|z||}.
\end{equation}
Next, we will use the special maximum principle as stated in  Lemma \ref{lem 4.6.6}. Let us choose a positive $\delta_1$ such that if
$|x |<\delta_1$, where $x,y$ are reals, and $\lambda =e^{x+iy}$, then
$$
\frac{1}{|1-|\lambda||} \le \frac{2}{|x|}.
$$
Assuming $z_0=e^{i\theta_0}$, we can choose $r>0$ such that $2r <\delta_1$, and   $g_n(e^\xi )$ is
analytic in $\xi \in B(i\theta_0 , 2r)$ for all $n$.
Therefore,  (\ref{3.10}) becomes
\begin{equation}\label{3.111}
\| g_n (e^\xi)\| \le \frac{2C}{|\Re \xi |}, \quad  \xi\in B(i\theta_0, 2r)\backslash i\R .
\end{equation}
By  Lemma \ref{lem 4.6.6}, (\ref{3.111}) yields that
\begin{equation}
\| g_n (e^\xi )\| \le \frac{4}{3}\cdot \frac{2C}{r}=\frac{8C}{3r}, \quad \mbox{for all} \ \xi \in B(i\theta_0, r).
\end{equation}
We can choose $\delta_2>0$ such that $B(z_0,\delta_2) \subset \{ e^\xi: \xi \in B(i\theta_0 ,4r)\} .$ The existence of such a positive $\delta_2$ is guaranteed because the function $z\mapsto e^z$ is a homeomorphism from a neighborhood of $i\theta $ onto a neighborhood of $z_0$.
This yields in particular that
\begin{equation}
\| g_n (z )\| \le  \frac{8C}{3r}, \quad \mbox{for all} \ z \in B(z_0,\delta_2).
\end{equation}
Since $C$ is independent of $n$ this proves (\ref{boundedness of g(z)}), that is, $g(z)$ is bounded  in $B(z_0,\delta_2)$. This is sufficient to deduce that $g(z)$ is holomorphic in $B(z_0,\delta_2)$.  In fact, for each $x^*\in \X^*$ and $k\in\N$ with $p_k: l^\infty (\X) \to \X$ defined as $p_k ( \{x(n)\}_{n=0}^\infty )=x(k)$ for all $\{ x(n)\}_{n=0}^\infty\in l^\infty (\X)$, the scalar function $x^*\circ p_k \circ g(\cdot )=x^*\circ g_k(\cdot )$ is holomorphic in $B(z_0,\delta_2)$. Since the family of continuous functionals $\{ x^*\circ p_k, x^*\in \X^*, k\in \N\}$ is separating,
by \cite[Theorem A.7, p. 459]{arebathieneu} the function $g(z)$ is holomorphic in $z\in B(z_0,\delta_2)$, that is, $R(z,S)x$ with $|z|>1$ has a holomorphic extension $g(z)$ in $B(z_0,\delta_2)$. Next, this yields that $R(z,\bar S)\bar x$ with $|z|>1$ has a holomorphic extension in $B(z_0,\delta_2)$. By Lemma \ref{lem 3.1} $z_0$ is not in $\sigma (x)$.
\end{proof}
\begin{definition}
We denote by $\sigma_Z ( x)$ the set of all points $\xi_0$ on $\Gamma$ such that the Z-transform of a sequence $x:=(x(n))\in l^\infty (\X)$ cannot be extended holomorphically to any neighborhood of $\xi_0$, and call this set the Z-spectrum of the sequence $x$.
\end{definition}
\begin{corollary}\label{cor 1}
Let $x:=(x(n))\in l^\infty (\X)$. Then
\begin{equation}
\sigma (x) \subset \sigma_Z(x).
\end{equation}
\end{corollary}
\begin{proof}
The corollary is an immediate consequence of Lemmas \ref{lem 3.1} and \ref{lem z-transform}
\end{proof}
\begin{remark}
In general, $\sigma (x)\not =\sigma_Z(x)$. Consider the following numerical sequence
$x:=(x(n))\in l^\infty (\R),$ where
\begin{eqnarray*}
 x(n):= \begin{cases} 0, \quad n=0\\
   1/n  , \quad n\in \N \backslash \{ 0\} .
   \end{cases}
 \end{eqnarray*}
Obviously, $\bar x =0$, so $\sigma (x)=\emptyset$. However, $ 1 \in \sigma_Z(x)$ because
$
\tilde{x}(z) =\sum_{j=1}^\infty  z^{-j} /j
$
cannot be extended holomorphically to a neighborhood of $1$.
\end{remark}
As an immediate consequence of Theorem \ref{the tec} and Corollary \ref{cor 1} we have
\begin{theorem}\label{the 2}
Let $(x(n))$ be a bounded sequence such that $\sigma_Z ( x)$  is countable, and let the following condition hold for
each $\xi_0\in \sigma _Z(x)$
\begin{equation}
\lim_{\lambda \downarrow \xi_0} (\lambda -\xi_0 )R(\lambda ,\bar
S)\bar x =0 .
\end{equation}
Then, $ \lim_{n\to\infty}x(n)=0.$
\end{theorem}

\section{Volterra Equations of Convolution Form}
This section will be devoted to various applications of Theorems \ref{the tec} and \ref{the 2} to Volterra equations of the form (\ref{eq}) and (\ref{eq1}).
\begin{definition}
We say that a point $z_0\in \Gamma$ is regular with respect to Eq. (\ref{eq}) if there exists a neighborhood of $z_0$ in the complex plane in which
the operator
\begin{equation}
(z-A-\tilde{B}(z))^{-1}
\end{equation}
exists as an element of $L(\X)$ for each fixed $z$, and  as a  function of $z$, is holomorphic. We denote by $\Sigma$ the set of all points in $\Gamma$ that are not regular with respect to (\ref{eq}).
\end{definition}
Let us consider a bounded solution $x:=(x(n))$ to Eq. (\ref{eq1}) under condition (\ref{bounded eq}). For each $|z|>1$ taking the Z-transform of both sides of Eq. (\ref{eq}) gives
\begin{equation}\label{3.2}
z\tilde{x}(z)-zx(0)= A\tilde{x}(z)+\tilde{B}(z)\tilde{x}(z)+\tilde{y}(z).
\end{equation}
For each $z_0\in \Gamma \backslash \Sigma$, by (\ref{3.2}) we have
\begin{equation}\label{2.5}
\tilde{x}(z)=(z -A-\tilde{B}(z))^{-1} (zx(0)+\tilde{y}(z)),
\end{equation}
where $z$ is in a neighborhood of $z_0$.
As an immediate consequence of Corollary \ref{cor 1} we have the following
\begin{lemma}\label{lem spec} Let $x:=(x(n))$ be a bounded solution of
equation (\ref{eq1}). Then,
\begin{equation}\label{spec estimate}
\sigma (x) \subset \sigma_Z(x)\subset   \Sigma \cup \sigma_Z(y).
\end{equation}
\end{lemma}
Recall that the resolvent sequence $\{ X(n)\}_{n=0}^\infty \subset L(\X)$ is defined as the operator solution of (\ref{eq}), that is,
\begin{equation}
X(n+1)=AX(n)+\sum_{j=0}^n B(n-j)X(j),
\end{equation}
such that $X(0)=I$ (the identity operator in $\X$). The existence and uniqueness of such a resolvent sequence can be easily proved. Every solution $(x(n))$ of (\ref{eq}) with initial value $x(0)$ can be represented as $x(n)=X(n)x(0)$ for all $n\in\N$.
\begin{example}
If $B(n)=0$ for all $n\in \N$, then $(X(n))=(A^n)$ is the resolvent sequence of (\ref{eq}).
\end{example}
\begin{definition}
Eq. (\ref{eq}) is said to be  asymptotically stable if
$
\lim_{n\to\infty} X(n) =0.
$
Eq. (\ref{eq}) or Eq. (\ref{eq1}) is said to be strongly asymptotically stable if for every solution $(x(n))$ of the equation $lim_{n\to\infty}x(n)=0$. 
A sequence $x:=(x(n)) $ is said to be asymptotically stable if $\lim_{n\to\infty}x(n)=0$.
\end{definition}
\begin{theorem}\label{the 4.6}
Let the set $\Sigma$ associated with Eq. (\ref{eq}) be countable.
Then, the following assertions hold:
\begin{enumerate}
\item
A bounded solution $x:=(x(n))$ of Eq. (\ref{eq}) is  asymptotically stable if for each $\xi_0 \in \Sigma$
\begin{equation}\label{4.7}
\lim_{\lambda \downarrow \xi_0} (\lambda -\xi_0 )R(\lambda ,
\bar S)  \bar x  =0 ;
\end{equation}
\item Eq. (\ref{eq}) is  asymptotically stable if $(X(n))$ is bounded, and for each $\xi_0 \in \Sigma$, the following limit holds in $l^\infty (L(\X))$
\begin{equation}\label{4.8}
\lim_{\lambda \downarrow \xi_0} (\lambda -\xi_0 )R(\lambda ,
\bar S)  \bar X   =0 .
\end{equation}
\end{enumerate}
\end{theorem}
\begin{proof}
This theorem is an immediate consequence of Theorem \ref{the 2} and Lemma \ref{lem spec}.
\end{proof}
\begin{remark}
\begin{enumerate}
\item
 As shown in \cite{min}, if $B(n)=0$ for all $n\in\N$, then (\ref{4.7}) follow from the condition that $R\sigma (A) \cap \Gamma=\emptyset$. In the general case (\ref{4.7}) follows from (\ref{erg cond of V}) below.
 \item $(X(n))$ is bounded if $\| A\| +\sum_{n=0}^\infty \| B(n)\|  \le 1$.
\item Condition (\ref{4.8}) can be easily checked in some interesting cases  as below.
 \end{enumerate}
 \end{remark}
 \begin{corollary}\label{cor 3}
Let $\Sigma =\emptyset$, that is, $(z-A-\tilde{B}(z))^{-1}$ exists and holomorphic in 
$\{ z\in \C: |z|\ge 1\}$. Assume further that $(X(n))$ is bounded. Then, Eq. (\ref{eq}) is asymptotically stable.
\end{corollary}
\begin{proof}
Since $\Sigma =\emptyset$, condition (\ref{4.8}) is automatically satisfied. Therefore, by Theorem \ref{the 4.6}, $\lim_{n\to\infty} X(n)=0$, that is, Eq. (\ref{eq}) is asymptotically stable.
\end{proof}

\begin{theorem}\label{the extension of K_T Thrm}
Assume that  $\Sigma\subset \{ 1\}$ and the resolvent sequence $(X(n))$ is bounded. Then
\begin{equation}\label{4.10}
\lim_{n\to\infty} [X(n+1)-X(n)]=0 .
\end{equation}
\end{theorem}
\begin{proof}
By Lemma \ref{lem spec} we have $\sigma (X) \subset \sigma_Z(X) \subset \Sigma \subset \{ 1\}$, so an elementary proof of this theorem can be taken from that of \cite[Theorem 2.1]{min} that actually shows that if a sequence $x:=(x(n))\in l^\infty (\X)$ has $\sigma (x) \subset \{ 1\}$, then $\lim_{n\to\infty} [x(n+1)-x(n)]=0$.
\end{proof}
\begin{remark}
Theorem \ref{the extension of K_T Thrm} is an extension of the Katznelson-Tzafriri Theorem \cite[Theorem 1]{kattza}. In fact, when $B(n)=0$ for all $n\in\N$, then $X(n)=A^n$. Therefore, the operator $A$ is power bounded if and only if $(X(n))\in l^\infty (L(\X))$.
\end{remark}

\bigskip
For homogeneous equations (\ref{eq}) the spectral estimate (\ref{spec estimate}), that involves Z-spectrum of a bounded solution, plays an important role. However, for inhomogeneous equations (\ref{eq1}) there seems to be some inconvenience. In fact, let us consider the very simple inhomogeneous equation of the form $x(n+1) = y(n)$ for $n\in\N$, that is, in this case $A=0, B(n)=0$ for all $n\in\N$. If we use the Z-spectra of $x$ and $y$ we will have the estimate $\sigma_Z(x)\subset \sigma_Z(y)$. For $x:=(x(n))$ to be asymptotically stable, that is, $x\in c_0(\X)$, we need $\sigma_Z(y)=\emptyset$. However, as noted above this condition is too strong on $y$ because for many $y:=(y(n))\in c_0(\X)$ one has $\sigma_Z(y)\not= \emptyset$. In order to avoid this below we present an alternative to (\ref{spec estimate}). Let us denote by $V$ the operator $l^\infty (\X)\ni x:=(x(n)) \mapsto Ax+B*x \in l^\infty (\X)$, and $B*$ the operator $l^\infty (\X) \ni x:=(x(n))\mapsto B*x \in l^\infty (\X)$. It is easy to check that under condition (\ref{bounded eq}), $V$ leaves $c_0(\X)$ invariant, so it induces an operator $\bar V$ in the quotient space $\Y=l^\infty (\X)/c_0(\X)$. Moreover, for each $x:=(x(n))\in l^\infty (\X)$ we have
\begin{eqnarray}
 S(B*x)-B*Sx &=&  \{ B(n+1)x(0)\}_{n=0}^\infty .  \label{100}
\end{eqnarray}
By(\ref{100}), in general, $SV\not= VS$ and $SB* \not= B*S$. However, under (\ref{bounded eq}), (\ref{100}) also shows that in the quotient space $\Y$, we have $\bar S\bar V=\overline{SV}=\overline{VS}=\bar V \bar S$ and $\bar S \ \overline{B*}=\overline{B*}\ \bar S$, that is, in $\Y$, $\bar V$ and $\overline{B*}$ commutes with $\bar S$.

\medskip
For any bounded operator $B\in L(\X)$ let us denote $\sigma_\Gamma (B) := \sigma (B) \cap \Gamma$. The following lemma is an alternative of Lemma \ref{lem spec}.
\begin{lemma}\label{lem strong spec estimate}
Let $x:=(x(n))$ and $y:=(y(n))$ be in $l^\infty (\X)$ such that $x$ is a solution of (\ref{eq1}). Then
\begin{equation}\label{strong spec estimate}
\sigma (x) \subset \sigma_\Gamma (\bar V) \cup \sigma (y).
\end{equation}
\end{lemma}
\begin{proof}
Since $\bar S$ commutes with $\bar V$ for each $|\lambda|\not=1$ the operator $R(\lambda ,\bar S)$ commutes with $\bar V$ as well. Therefore, since in the operator form Eq. (\ref{eq1}) can be re-written as $Sx=Vx+y$ and for each $|\lambda|\not= 1$ the identity $ \lambda R(\lambda ,\bar S ) \bar x
-\bar x = R(\lambda ,\bar S )\bar S \bar x  $   holds
we have
\begin{eqnarray*}
\lambda R(\lambda ,\bar S ) \bar x
-\bar x&=& R(\lambda ,\bar S) \bar S \bar x
 =
R(\lambda ,\bar S) \bar V\bar x+  R(\lambda ,\bar S) \bar y\\
&=&  \bar VR(\lambda ,\bar S)\bar x+   R(\lambda ,\bar S)\bar y .
\end{eqnarray*}
Therefore, for each $|\lambda|\not= 1$,
\begin{eqnarray*}
(\lambda - \bar V) R(\lambda ,\bar S)\bar x &=&   \bar x + R(\lambda ,\bar S)  \bar y .
\end{eqnarray*}
Let $z_0\in \Gamma$ such that $z_0\not\in \sigma (\bar V)$ and $z_0\not\in \sigma (y)$. Then, in a small neighborhood of $z_0$, by definition $R(\lambda ,\bar S)  \bar y$ has a holomorphic extension $h(\lambda )$, and $(\lambda -\bar V)^{-1}$ exists and is holomorphic in $\lambda$, so for $\lambda$ in such a neighborhood of $z_0$,
\begin{eqnarray}\label{resolvent of S}
 R(\lambda ,\bar S)\bar x &=&  ( \lambda - \bar V )^{-1} \bar x + ( \lambda - \bar V )^{-1} h(\lambda ) .
\end{eqnarray}
This shows that $z_0\not\in \sigma (x)$ and yields (\ref{strong spec estimate}).
\end{proof}

\begin{theorem}\label{the Minh thrm on stability}
Let  $\sigma_\Gamma (\bar V)$ be countable and let the following conditions be satisfied:
\begin{equation}\label{erg cond of V}
R\sigma (\bar V) \cap \Gamma =\emptyset .
\end{equation}
Then,
\begin{enumerate}
\item Eq. (\ref{eq}) is strongly asymptotically stable provided that
\begin{equation}\label{contraction operator}
\| A \| + \sum_{n=0}^\infty \| B(n) \| \le 1 ,
\end{equation}
\item Every bounded solution of Eq. (\ref{eq1}) with $(y(n))\in c_0(\X)$ is asymptotically stable;
\item Every bounded solution of Eq. (\ref{eq1}) with $(y(n))\in AAP(\N,\X)$ is asymptotically almost periodic.
\end{enumerate}
\end{theorem}
\begin{proof}
(i) First, by (\ref{contraction operator}) the operator $V$ is a contraction, so is the induced operator $\bar V$. Therefore, every solution of (\ref{eq}) is bounded, and $\sigma (\bar V)$ is part of the closed unit disk $\bar \D$. Next, if $|\lambda |>1$, by (\ref{resolvent of S}) (with $h(\lambda )=0$) yields
\begin{equation}\label{1001}
 R(\lambda ,\bar S)\bar x  =  R(\lambda , \bar V) \bar x .
\end{equation}
Notice that \cite[Remark 3.4]{min} actually shows that if $T$ is any bounded operator in a Banach space with $R\sigma (T)\cap \Gamma =\emptyset$, then, for each $\xi_0\in \Gamma$ and $x\in \X$,  $\lim_{\lambda \to\xi_0}R(\lambda ,T)x=0$. Therefore,
by (\ref{1001}), condition (\ref{erg cond of V}) yields that for each $\xi_0\in \sigma (x) \subset \sigma _\Gamma (\bar V)$
\begin{eqnarray}
0\le \lim_{\lambda \downarrow \xi_0} \| (\lambda -\xi_0 )R(\lambda ,\bar
S)\bar x \| &=& \lim_{\lambda \downarrow \xi_0} \| (\lambda -\xi_0 )R(\lambda ,\bar
V)\bar x \| =0.
\end{eqnarray}
Finally, the assertion of the theorem follows from that of Theorem \ref{the tec}.

\medskip
(ii) The proof is similar to (i), so details are omitted.

\medskip
(iii) If we replace $\Y$ by the quotient space $ l^\infty (\X)/AAP(\N,\X)$ and repeat the above construction and proof of (i) we can prove (iii). Therefore, the details are omitted.
\end{proof}
\begin{remark}
When $B(n)=0$ for all $n\in \N$, it is easy to see that $\sigma (\bar V) \subset \sigma (V) \subset \sigma (A)$. Accordingly, part i) of Theorem \ref{the Minh thrm on stability} is a version for Volterra equations of the Arendt-Batty-Ljubich-Vu Theorem \cite[Theorem 5.1]{arebat3}. It is clear that parts (ii) and (iii) are extensions of a theorem due to Y. Katznelson, L. Tzafriri (see \cite[Theorem 1]{kattza}) for individual orbits of discrete systems.
\end{remark}

Below we will demonstrate further advantages of using the operator $V$ determined by the Volterra equations (\ref{eq}) in the quotient space $l^\infty (\X)/AAP(\N,\X)$. In fact, we will address the question as how the "input" $y:=(y(n))$ in Eq. (\ref{eq1}) can control the "output" $x:=(x(n))$ as a solution to (\ref{eq1}).
\begin{lemma}\label{lem 4.13}
Let $\F$ be a translation-bi-invariant closed subspace of $l^\infty (\X)$. Moreover, assume that ${\cal F}$ is invariant under the operator $\bar V$. Then, for every bounded solution $x:=(x(n))\in l^\infty (\X)$ of (\ref{eq1}), the following spectral estimate holds
\begin{equation}\label{new spec estimate}
\sigma_\F (y)  \subset  \sigma_\F (x) .
\end{equation}
\end{lemma}
\begin{proof}
Since $x$ is a bounded solution of (\ref{eq1}) for $|\lambda |\not= 1$ we have
$$
R(\lambda , \bar S)\bar S \bar x = R(\lambda , \bar S)\bar V \bar x +R(\lambda , \bar S)\bar y.
$$
As $\bar S$ commutes with $\bar V$, for $|\lambda |\not= 1$ and $R(\lambda , \bar S)\bar S\bar x= \lambda R(\lambda , \bar S) \bar x-\bar x $ this yields
\begin{eqnarray*}
R(\lambda , \bar S)\bar S\bar x &=&  \lambda R(\lambda , \bar S) \bar x-\bar x \\
&=& \bar V R(\lambda , \bar S)\bar x +R(\lambda , \bar S)\bar y,
\end{eqnarray*}
so we have
$$
(\lambda -\bar V) R(\lambda , \bar S)\bar x -  \bar x = R(\lambda , \bar S)\bar y.
$$
Therefore, $R(\lambda , \bar S)\bar y$ can be extended to a holomorphic function in a neighborhood of each $\lambda_0\in \Gamma$ such that $\lambda_0\not\in \sigma _\F (x)$. And thus, (\ref{new spec estimate}) follows.
\end{proof}

\begin{theorem}
Eq. (\ref{eq1}) has no asymptotically stable solution if $y\not\in c_0 (\X)$. Similarly, it has no asymptotically almost periodic solutions if $y\not\in AAP(\N, \X)$.
\end{theorem}
\begin{proof}
Every asymptotically stable solution $x$ is in $c_0(\X)$, so its spectrum $\sigma (x)$ must be empty. By Lemma \ref{lem 4.13} this is impossible if $y\not\in c_0(\X)$. Similarly, we can prove the second part of the assertion.
\end{proof}

%%%%%%%%%%%%%%%%%%%%%%% REFERENCES %%%%%%%%%%%%%%%%%%%%%%%%%%
\bibliographystyle{amsplain}

\end{document}